\documentclass[12pt]{article}

\usepackage[a4paper,margin=1in]{geometry}
\usepackage{amsmath,amssymb,amsthm,mathtools}
\usepackage{booktabs,array,multirow}
\usepackage{tikz}
\usetikzlibrary{arrows.meta,positioning,calc,shapes.geometric,fit,backgrounds,shadows}
\usepackage{caption}
\usepackage{subcaption}
\usepackage{enumitem}
\usepackage{xcolor,cite,float}
\usepackage{hyperref}
\hypersetup{colorlinks=true,linkcolor=blue,citecolor=blue,urlcolor=blue}
\allowdisplaybreaks
\renewenvironment{proof}{{\bf \noindent Proof.}}{\qed}
\everymath{\displaystyle}

\newtheorem{theorem}{Theorem}[section]
\newtheorem{lemma}[theorem]{Lemma}
\newtheorem{proposition}[theorem]{Proposition}
\newtheorem{corollary}[theorem]{Corollary}

\newtheorem{problem}[theorem]{ Problem}
\theoremstyle{definition}

\newtheorem{example}[theorem]{Example}
\theoremstyle{remark}

\title{On the vertex connectivity of the weakly zero-divisor graph of commutative rings}
 \author{Mohd Shariq$^{a}$  and Jitender Kumar$^{b, *}$     \\
	 \small $^{a,b}$Department of Mathematics, Birla Institute of Technology and Science Pilani, Pilani-333031, India \\ 
     $^{a}$\texttt{shariqamu90@gmail.com}, $^{b}$\texttt{jitenderarora09@gmail.com}  \\
 }

\date{}
\begin{document}
\maketitle

 \begin{abstract}
      The weakly zero-divisor graph $W\Gamma(R)$ of a commutative ring $R$ is the simple undirected graph whose vertices are nonzero zero-divisors of $R$, and two distinct vertices $x$, $y$ are adjacent if and only if there exists $w\in {\rm ann}(x)$ and $ z\in {\rm ann}(y)$ such that $wz =0$. In this paper, first we prove that the vertex connectivity of $W\Gamma(R)$ is equal to its minimum degree, where  $R$ is either an Artinian ring or reduced ring. For any finite ring $R$, we obtain the vertex connectivity of $W\Gamma(R)$. Moreover, this paper characterizes all the vertices that attain the minimum degree of $W\Gamma(R)$.

      
    \end{abstract}

*  Corresponding author
 \section{Introduction}
 Several graphs associated with algebraic structures have been extensively investigated in the literature. An overview of the various graph constructions arising from commutative rings can be found in the book \cite{Anderson2021Graphs}. In the setting of commutative rings, the classical zero-divisor graph translates the algebraic equation $xy = 0$ into a graph-theoretic adjacency relation. The zero-divisor graph originated in Beck's \cite{Beck1988} study of coloring problems for commutative rings, in which the vertices are all the elements of the ring. The definition was later modified by Anderson and Livingston \cite{AndersonLivingston1999}, yielding the now-standard zero-divisor graph whose vertices are precisely the nonzero zero-divisors of the ring. Since then, several graphs have been associated with commutative rings. These graphs help us to understand the relationship between the algebraic properties of a ring $R$ and the graph-theoretical properties of associated graphs.  
 
 For $x\in R$, the annihilator of $x$ is defined by ${\rm ann}(x)=\{r\in R\mid rx=0\}$.  Nikmehr et al.  \cite{nikmehr2021weakly} introduced the  \emph{weakly zero-divisor graph} $W\Gamma(R)$ of the ring $R$ is the simple undirected graph whose vertex set is the set of all nonzero zero-divisors of the ring $R$, and two vertices $x, y$ are adjacent if and only if there exist $w\in {\rm ann}(x)$ and $z \in {\rm ann}(y)$ such that $wz =0$. One can observe that the zero-divisor graph is a spanning subgraph of the weakly zero-divisor graph. The graph $W\Gamma(R)$ has been studied from several viewpoints.
Nikmehr et al.  \cite{nikmehr2021weakly} explored the relationship between zero-divisor graphs and weakly zero-divisor graphs. Moreover, they examined several basic properties including completeness, girth, clique number, and vertex chromatic number of $W\Gamma(R)$.
Rehman et al. \cite{ur2024planarity} classified all the commutative rings $R$ for which the weakly zero-divisor graph is a star graph, a unicyclic graph, a tree and a split graph. Furthermore, they identified all the rings $R$ for which $W\Gamma(R)$ is planar, toroidal, bitoroidal and of crosscap at most two, respectively. Visweswaran \cite{visweswaran2025some} studied the complement of the weakly zero-divisor graph of a reduced ring. Moreover, they investigated its connectedness, girth and number of all connected components.
Kumar et al. \cite{shariq2023laplacianspectrum} investigated the Laplacian eigenvalues of the weakly zero-divisor graph of the ring $\mathbb{Z}_n$. Additionally, they showed that the weakly zero-divisor graph of the ring $\mathbb{Z}_n$ is Laplacian integral for an arbitrary $n$.
Nadeem et al. \cite{rehman2024randic} investigated the Randic spectrum of the weakly zero-divisor graph of the ring $\mathbb{Z}_n$.  Other aspects of weakly zero-divisor graphs have also been studied in the literature, see \cite{mozumder2025exploring, rehman2025signless,shariq2025sombor,shylla2025laplacian} and the references therein.
 
The Notion of vertex connectivity of the zero-divisor graph has been investigated by various researchers. Extell et al. \cite{axtell2011cut} examine the cut vertices of the zero divisor graph of the finite commutative rings. Moreover, they provide a partial classification of the rings in which they appear.
  Redmond \cite{redmond2012cut} continues to examine cut vertices in the zero-divisor graphs of commutative rings with unity. Furthermore, they investigated the degree one vertices of zero-divisor graphs. B. {Cot{\'e}  et al. \cite{cote2011cut} examine the minimal cut sets in zero-divisor graphs of finite non-local commutative rings with identity. Reza Akhtar et al. \cite{akhtar2016connectivity} studied the vertex-connectivity and edge-connectivity of the zero-divisor graph associated with a finite commutative ring $R$. They showed that the edge-connectivity of the zero-divisor graph equals its minimum degree, and that the vertex-connectivity also equals the minimum degree when $R$ is non-local. Sriparna et al. \cite{chattopadhyay2024vertex} determine the vertex connectivity of the zero-divisor graph, where $R$ is a local principal ideal ring or $R$ is a finite direct product of local principal ideal rings. Moreover, they characterize the vertices of minimum degree and the minimum cut-sets of the zero-divisor graph of the ring $R$. It is natural to investigate the vertex connectivity of the graphs which contains zero-divisor graph as a subgraph.
  In this manuscript, we aim to study the vertex connectivity of the weakly zero-divisor graph of the commutative rings.

In Section 3, we study the weakly zero-divisor graph $W\Gamma(R)$, where $R$ is a reduced ring having  finite number of minimal prime ideals. For the reduced ring $R$, the set of zero-divisors is nothing but the union of all the minimal prime ideals of $R$. These minimal prime ideals play a fundamental role in obtaining the vertex connectivity of $W\Gamma(R)$. The distinct sets of vertices that contain the elements of one prime ideal but not the other minimal prime ideals are useful in determining the cut sets, the minimum cut set and the vertices of the minimum degree.
 In Section 4, we use the decomposition of Artinian ring $R$ into fields and non-reduced Artinian local rings to study the structure of $W\Gamma(R)$. Consequently, we construct vertex sets that determine the corresponding cut sets, and among these sets, the sets of maximum cardinality help us to determine vertex connectivity and the vertices of minimum degree. 


 \section{Preliminaries}
 Let $\Gamma$ be a simple graph with vertex set $V(\Gamma)$ and edge set $E(\Gamma)$. Two adjacent vertices $x$, $y$ in the graph $\Gamma$ are denoted by $x\sim y$. A path between two vertices $x$ and $y$ of a graph $\Gamma$ is an ordered sequence of distinct vertices $x_1,x_2,\ldots,x_{k}$ such that $x_ix_{i+1}\in E(\Gamma)$ for every $i\in\{1,\ldots,k\}$.
 For an arbitrary connected graph $\Gamma$, a subset 
$A \subseteq V(\Gamma)$ is called a \emph{cut-set} if there exist distinct 
vertices $x,y \in V(\Gamma)\setminus A$ such that every path from 
$x$ to $y$ in $\Gamma$ contains at least one vertex of $A$, and no 
proper subset of $A$ possesses the same property. A cut-set $A$ is said to be \emph{minimum} if $|A| \leq |B|$ for every cut-set $B$ of $\Gamma$. Moreover, if $V(\Gamma)$ can be expressed as a union of two non-empty disjoint subsets $A$ and $B$ such that no edge of $\Gamma$ from a vertex of $A$ to a vertex of $B$, then $(A, B)$ is called a separation of $\Gamma$. Note that the graph $\Gamma$ is disconnected if and only if it admits a separation of $\Gamma$. The \emph{neighbourhood} ${N(x)}$ of a vertex $x\in V(\Gamma)$  is the set of vertices adjacent to the vertex $x$ and the \emph{degree} of a vertex $x \in V(\Gamma)$, denoted by $\deg(x)$, is 
defined as the cardinality of its neighbourhood. The\emph{ minimum degree} of $\Gamma$ is defined by
$\delta(\Gamma)=\text{min}\{\deg(v)\mid v\in V(\Gamma)\}$.
For a graph $\Gamma$, the \emph{vertex-connectivity} $\kappa(\Gamma)$ is the cardinality of a minimum cut set of the graph $\Gamma$.

The following well-known inequality relates the vertex connectivity, edge connectivity, and minimum degree of a graph $\Gamma$.
\begin{proposition}[{\cite[Theorem 4.1.9]{west2001introduction}}]
\label{minimumdegree+vertex}
For any graph $\Gamma$, we have
$$
\kappa(\Gamma)\leq \lambda(\Gamma)\leq \delta(\Gamma).
$$
\end{proposition}
Throughout this paper, we assume that all the rings are commutative with identity. An element $x \in R$ is called a \emph{zero-divisor} if there exists a nonzero element $y \in R$ such that $xy=0$, and an element $u\in R$ is called a \emph{unit} if there exists
$v\in R$ such that
$uv=1$. An element $x \in R$ is said to be \emph{nilpotent} if $x^m=0$ for some positive integer $m$. If $x$ is nilpotent, then the least positive integer $k$ for which $x^k=0$ is called the 
\emph{index of nilpotence} of $x$. Every nilpotent element of $R$ is a zero-divisor. For each $r \in R$, the set
${\mathrm{ann}}(r)=\{x \in R \mid rx=0\}
$
forms an ideal of $R$, called the annihilator of $r$.
We denote the set of all zero-divisors of $R$ by $Z(R)$, the set of all units of $R$ by $U(R)$, the set of all prime ideals of a ring $R$ by $\mathrm{Spec}(R)$, the set of all minimal prime ideals of a ring $R$ by $\mathrm{Min}(R)$ and the set of all nilpotent elements of $R$ by $\mathrm{Nil}(R)$. For any subset $S \subseteq R$, we define $S^* = S \setminus \{0\}$. A ring $R$ is called \emph{local} if it has a unique maximal ideal. The ring $R$ is said to be \emph{reduced} if it contains no nonzero nilpotent elements. For more notation on graphs and rings, refer to \cite{atiyah2018introduction, west2001introduction}. Now we recall the following classical structure theorem for Artinian rings.
\begin{theorem} [{\cite[Theorem 3.1.4]{bini2002finite}}]
  An Artinian ring $R$ can be written uniquely (up to isomorphism) as 
$R \cong R_1 \times R_2 \times \cdots \times R_n$,
where each $R_i$ is an Artinian local ring.
\end{theorem}
 The next result records the completeness of the weakly zero-divisor graph $W\Gamma(R)$ of an Artinian ring $R$.
 \begin{theorem}[{\cite[Theorem 2.6]{nikmehr2021weakly}}]\label{nikmehr2021weaklycmpt}
Let $R$ be an Artinian ring. Then $W\Gamma(R)$ is a complete graph if and only if one of the following statements holds:
\begin{enumerate}[label=\textup{(\roman*)}]
    \item $R \cong \mathbb{Z}_2 \times \cdots \times \mathbb{Z}_2$.
    \item $R \cong R_1 \times \cdots \times R_m$, where $R_i$ is a non-reduced Artinian local ring for every $1 \leq i \leq m$.
\end{enumerate}
\end{theorem}
\section{Reduced rings} 
For an ideal $I$ of a ring $R$, define $V(I)=\{\mathfrak{p}\in Spec(R)\mid I\subseteq \mathfrak{p}\}$. The following lemma is useful in the sequel.
\begin{lemma}[{\cite[Theorem 2.6]{visweswaran2025some}}]\label{edgejiolemmaREDUCE}
    Let $R$ be a reduced ring with $|\mathrm{Min}(R)|=n\geq 2.$  Then $x\sim y$ in $W\Gamma(R)$ if and only if either $V(Rx)\cap \mathrm{Min}(R)\neq V(Ry)\cap \mathrm{Min}(R)$ or   $ |V(Rx)\cap \mathrm{Min}(R)|\geq 2$.
\end{lemma}
For each $1\leq i\leq n$, define
$\gamma_i:=\mathfrak{p_i}\setminus\bigcup_{j=1,j\neq i}\mathfrak{p_j}$. Let $\alpha_{\max}:=\max\{|\gamma_i|:1\leq i\leq n\}$ and
$S:=\{j: |\gamma_j|=\alpha_{max}\}$. 

\begin{lemma}\label{vertesetcondition}
Let $R$ be a reduced ring with $ |\mathrm{Min}(R)|=n\geq 2$. Then $|V(Rx)\cap  \mathrm{Min}(R)|=1$ if and only if $x\in \gamma_i$. 
\end{lemma}
\begin{proof} Suppose $|V(Rx)\cap  \mathrm{Min}(R)|=1$. Without loss of generality, assume that $V(Rx)\cap  \mathrm{Min}(R)=\{\mathfrak{p_i}\}$. Then $Rx\subseteq\mathfrak{p_i}$ and so $x\in \mathfrak{p_i}$. To end the proof, we show that $x\not\in \mathfrak{p_j}$ for all $j\in\{1,2,\dots,n\}\setminus\{i\}$. Suppose $x\in\mathfrak{p_j}$ for some  $j\in\{1,2,\dots,n\}\setminus\{i\}$. Then $Rx\subseteq\mathfrak{p_j}$ and so $\mathfrak{p_j}\in V(Rx)\cap  \mathrm{Min}(R)$ for some  $j\in\{1,2,\dots,n\}\setminus\{i\}$. It follows that $|V(Rx)\cap  \mathrm{Min}(R)|\geq 2$, a contradiction. Thus,  $x\in\gamma_i$. Let  $x\in \gamma_i$. Then $x\in \mathfrak{p_i}$ for some $i\in\{1,2,\ldots,n\}$ but $x\not\in\mathfrak{p_j}$ for all $j$, where $j\in\{1,2,\ldots,n \}\setminus\{i\}$. Clearly, $Rx\subseteq \mathfrak{p_i}$ but  $Rx\not\subseteq \mathfrak{p_j}$ for all $j$. Therefore,  $V(Rx)\cap \mathrm{Min}(R)=\{\mathfrak{p_i}\}$. This completes our proof.
\end{proof}

\begin{lemma}
    Let $R$ be reduced ring with $|\mathrm{Min}(R)|=n\geq2$. Then  $W\Gamma(R)$ is a complete graph if and only if $\alpha_{max}=1$. 
\end{lemma}
\begin{proof}
Let $W\Gamma(R)$ be a complete graph. Suppose $\alpha_{max} \geq 2$. Then $|\gamma_i| \geq 2$ for some $i \in S$. Let $x,y \in \gamma_i$. Then $V(Rx) \cap \mathrm{Min}(R)= \{\mathfrak{p_i}\}= V(Ry) \cap \mathrm{Min}(R).$ Moreover, $\operatorname{ann}(x) = \cap_{{j=1,j\neq i}}^{n}\mathfrak{p_j}= \operatorname{ann}(y).$ Let $x' \in \operatorname{ann}(x)$ and $y' \in \operatorname{ann}(y)$. It follows that $x',y' \notin \mathfrak{p_i}$ and so $x'y' \neq 0,$ a contradiction. Conversely, let  $\alpha_{max}=1$. Then $|\gamma_i|=1$ for all $i\in \{1,2,\ldots,n\}$. Let $x,y\in V(W\Gamma(R))$. Then $x\in \mathfrak{p_i}$ and $y\in \mathfrak{p_j}$ for some $i,j\in \{1,2,\ldots,n\}$. It follows that either $x\in \gamma_i$ or  $x\not\in \gamma_i$. First, suppose that  $x\in \gamma_i$. Then $V(Rx)\cap  \mathrm{Min}(R)=\{\mathfrak{p_i}\}$. If $y\in \gamma_j$,  then $V(Ry)\cap  \mathrm{Min}(R)=\{\mathfrak{p_j}\}$. Therefore, by Lemma \ref{edgejiolemmaREDUCE}, we have $x\sim y$. If $y\not\in \gamma_j$, then by Lemma \ref{vertesetcondition}, we have $|V(Ry)\cap  \mathrm{Min}(R)|\geq2$ and so by  Lemma \ref{edgejiolemmaREDUCE},  $x\sim y$. We may now suppose that $x\not\in \gamma_i$. Then $|V(Rx)\cap  \mathrm{Min}(R)|\geq2$. If $y\in \gamma_j$ or  $y\not\in \gamma_j$, then by Lemma \ref{edgejiolemmaREDUCE} and \ref{vertesetcondition}, we have $x\sim y$. Hence,  $W\Gamma(R)$ is a complete graph.
\end{proof}

Suppose $|\mathrm{Min}(R)|=2$ and let $\mathrm{Min}(R)=\{\mathfrak{p_1},\mathfrak{p_2}\}$. As $\mathrm{Nil}(R)=(0)$, by {\cite[Proposition 1.8]{atiyah2018introduction}}, we have $\mathfrak{p_1}\cap \mathfrak{p_2}=0$ and so $Z(R)^*=\mathfrak{p_1}\cup \mathfrak{p_2}$. Therefore,
$W\Gamma(R)$ is a complete bipartite graph with vertex set $V(W\Gamma(R))=V_1\cup V_2$, where $V_1=\mathfrak{p_1}\setminus\{0\}$ and $V_2=\mathfrak{p_2}\setminus\{0\}$. Hence, either $\mathfrak{p_1}\setminus\{0\}$ or $\mathfrak{p_2}\setminus\{0\}$ is a minimum cut-set of $W\Gamma(R)$. Consequently,
$$
\kappa(W\Gamma(R))
=\min\{|\mathfrak{p_1}\setminus\{0\}|,\ |\mathfrak{p_2}\setminus\{0\}|\}
=\delta(W\Gamma(R)).
$$
In the remainder of this section, we assume that $\alpha_{max}\geq 2$ and $|\mathrm{Min}(R)|\geq3$.

 
\begin{theorem}
    Let $R$ be a reduced ring with $ |\mathrm{Min}(R)|=n\geq 3$. Then $\kappa(W\Gamma(R))=\delta(W\Gamma(R))$
\end{theorem}
\begin{proof} Let $ \mathrm{Min}(R)=\{\mathfrak{p_i}\mid i\in\{1,2,\ldots,n\}\}$. As $\mathrm{Nil}(R)=(0)$, by {\cite[Proposition 1.8]{atiyah2018introduction}}, we have $\cap_{i=1}^{n}\mathfrak{p_i}=0$. Note that $Z(R)=\cup_{i=1}^{n}\mathfrak{p_i}$. By Proposition \ref{minimumdegree+vertex}, it remains to prove that $\kappa(W\Gamma(R)) \geq \delta(W\Gamma(R))$. To this end, let $A \subsetneq  V(W\Gamma(R))$ such that $|A|< \delta(W\Gamma(R))$. We show that the graph  $G = W\Gamma(R) - A$ is connected. 

Let $x,y\in G$. Then we provide at least one path between $x$ and $y$.  Since $V(G)\subseteq Z(R)^*$, it follows that $x\in \mathfrak{p}_i$
and  $y\in\mathfrak{p}_j$ for some $i,j\in \{1,2,\ldots,n\}$.
Consider a vertex  $w\in\gamma_i$. Then by Lemma \ref{vertesetcondition},  we have $V(Rw)\cap \mathrm{Min}(R)=\{\mathfrak{p_i}\}$. Suppose $a\in \gamma_i$. Then $V(Ra)\cap \mathrm{Min}(R)=\{\mathfrak{p_i}\}$. Moreover,  ${\rm ann}(w)=\cap_{j=1,j\neq i}^{n}\mathfrak{p_j}={\rm ann}(a)$.  Let $w'\in {\rm ann}(w)$ and  $a'\in {\rm ann}(a)$. Then we have $w',a'\not\in \mathfrak{p_i}$ and so  $w'a'\neq0$. Thus, $a\not\sim w$. Therefore, $N(w)\subseteq Z(R)^*\setminus\{\gamma_i\}$. It follows that $| Z(R)^*\setminus\{\gamma_i\}|\geq \delta(W\Gamma(R))> |A|$ and so  $G$ must contain a vertex $u\in N(w)$. By Lemma \ref{edgejiolemmaREDUCE}, observe that $N(w)\subseteq N(x)$. Therefore $u\in N(x)$.
Similarly, for $y\in \mathfrak{p_j}$ for some $j\in \{1,2,\ldots,n\}$, we obtain a vertex $z\in \gamma_j$ such that $N(z)\subseteq Z(R)^*\setminus\{\gamma_j\}$. Indeed, $G$ must contains a vertex  $v\in N(y)$. 

Let $u\in \mathfrak{p}_r$ and $v\in  \mathfrak{p}_m$ for some
$ {r,m}\in\{1,\ldots,n\}$. If either  $u\in \mathfrak{p}_s$ for some $s\in \{1,2,\ldots,n\}\setminus\{r\}$ or $v\in \mathfrak{p}_t$ for some $t\in \{1,2,\ldots,n\}\setminus\{m\}$, then either 
$\left|V(Ru)\cap\operatorname{Min}(R)\right|\geq 2$ or $\left|V(Rv)\cap\operatorname{Min}(R)\right|\geq 2$ and so by Lemma \ref{edgejiolemmaREDUCE}, we have $u\sim v$.

Let $u\notin \mathfrak{p}_s$ for all $s\in\{1,\ldots,n\}$
and $v\notin  \mathfrak{p}_t$ for all $t\in\{1,\ldots,n\}$. 
If $r\neq m$, then $V(Ru)\cap\operatorname{Min}(R)\neq V(Rv)\cap\operatorname{Min}(R)$. Therefore, by Lemma \ref{edgejiolemmaREDUCE}, we have $u\sim v$. If $r=m$, then $V(Ru)\cap\operatorname{Min}(R)=V(Rv)\cap\operatorname{Min}(R)$. Since $u\in N(x)$ and $v\in N(y)$, by Lemma \ref{edgejiolemmaREDUCE}, we have $V(Ru)\cap\operatorname{Min}(R)\neq V(Rx)\cap\operatorname{Min}(R)$ or $\left|V(Ru)\cap\operatorname{Min}(R)\right|\geq 2$ and $V(Ry)\cap\operatorname{Min}(R)\neq V(Rv)\cap\operatorname{Min}(R)$ or $\left|V(Rv)\cap\operatorname{Min}(R)\right|\geq 2$.
As $V(Ru)\cap\operatorname{Min}(R)= V(Rv)\cap\operatorname{Min}(R)$, it follows that $V(Ry)\cap\operatorname{Min}(R)\neq V(Ru)\cap\operatorname{Min}(R)$. Therefore, by Lemma \ref{edgejiolemmaREDUCE}, we get $ u\sim y$. If $\left|V(Ru)\cap\operatorname{Min}(R)\right|\geq 2$, then by Lemma \ref{edgejiolemmaREDUCE} , we obtain  $ u\sim y$. Therefore,  $ x\sim u\sim y$ is a path in $G$.
\end{proof}

\begin{theorem}\label{cutsetinreduced} 
Let $R$ be a reduced ring such that $ |\mathrm{Min}(R)|=n\geq 3$. Then, for each  $1\leq i\leq n$, the set $Z(R)^{*}\setminus\{\gamma_i\}$ is a cut set of  $W\Gamma(R)$. 
\end{theorem}
\begin{proof} Let $ \mathrm{Min}(R)=\{\mathfrak{p_i}\mid i\in\{1,2,\ldots,n\}\}$. As $\mathrm{Nil}(R)=(0)$, by {\cite[Proposition 1.8]{atiyah2018introduction}}, we have $\cap_{i=1}^{n}\mathfrak{p_i}=0$. Note that $Z(R)=\cup_{i=1}^{n}\mathfrak{p_i}$. Now, let $\Gamma_i$ be the subgraph induced by the set $\gamma_i$. Then, we claim that $\Gamma_i$ 
    is the null graph. Let $x,y\in \gamma_i$. Clearly, $|V(Rx)\cap \mathrm{Min}(R)|=1=|V(Ry)\cap  \mathrm{Min}(R)|$.
     Then by Lemma \ref{vertesetcondition}, we obtain $V(Rx)\cap \mathrm{Min}(R)=\{\mathfrak{p_i}\}$. Observe that ${\rm ann}(x)={\rm ann}(y)=\cap_{j=1,j\neq i}^{n}\mathfrak{p_j}$. Let $x'\in {\rm ann}(x)$ and  $y'\in {\rm ann}(y)$. Then we have $x',y'\not\in \mathfrak{p_i}$ and so  $x'y'\neq0$. Thus, $\Gamma_i$ is a null graph. Now, define a set $T=\{x\in Z(R): |V(Rx)\cap  \mathrm{Min}(R)|\geq 2 \}$. In view of Lemma \ref{vertesetcondition}, we have $Z(R) ^*$ = $ T+\cup_{i=1}^{n} \gamma_i$. Note that by Lemma \ref{edgejiolemmaREDUCE}, the graph   $W\Gamma(T) $ induced by the set $T$ is a complete. Let $z\in T$. Then $V(Rx)\cap \mathrm{Min}\neq V(Rz)\cap  \mathrm{Min}(R)$ and so  by Lemma \ref{edgejiolemmaREDUCE},  $x\sim z$. Now, let  $w\in \gamma_j$, $(j\neq i)$. Then  $V(Rx)\cap \mathrm{Min}(R)\neq V(Rw)\cap \mathrm{Min}(R)$ and so $x\sim w$. Thus, $ T+\cup_{j=1,j\neq i}^{n}\gamma_j \subseteq N(x)$. Since  $\Gamma_i$ is a null graph, with only neighbourhood set $T+\cup_{j=1,j\neq i}^{n}\gamma_j$ and so the set $T+\cup_{j=1,j\neq i}^{n}\gamma_j$  isolate the vertices of graph $\Gamma_i$.  To prove the minimality of the set $T+\cup_{j=1,j\neq i}^{n}\gamma_j$, let $A\subsetneq T+\cup_{j=1,j\neq i}^{n}\gamma_j$ be the cut set of the graph $W\Gamma(R)$. Then $A$ separates the graph $W\Gamma(R)$ into at least two parts $X$ and $Y$. Let $x\in X$ and $y\in Y$. 
    
First suppose $x\in \gamma_i$ and $y\in \gamma_j$ for some $i,j\in \{1,2,\dots,n\}$.
If  $i\neq j$, then $V(Rx)\cap  \mathrm{Min}(R)\neq V(Ry)\cap  \mathrm{Min}(R)$, so by Lemma \ref{edgejiolemmaREDUCE}, $x\sim y$,  which is not possible. If $ i=j$, then $x,y\in \gamma_i$. By using the similar argument used earlier, we have $x\not\sim y$ in $W\Gamma(R)$.


 For any  $z\in T$,   by Lemma \ref{edgejiolemmaREDUCE}, note that $x\sim z\sim  y$. It follows that $T\subseteq A$. Now let $w\in \gamma_k$ for some $k\in\{1,2,\ldots,n\}\setminus\{i\}$. Then ${\rm ann}(w)=\cap_{l=1,l\neq k}^{n}\mathfrak{p_l}$. Since ${\rm ann}(x)={\rm ann}(y)=\cap_{j=1,j\neq i}^{n}\mathfrak{p_j}$. Let $x'\in \cap_{j=1,j\neq i}^{n}\mathfrak{p_j}= {\rm ann}(x)$ and  $w'\in\cap_{l=1,l\neq k}^{n}\mathfrak{p_l}= {\rm ann}(w)$. Then $x'w'=0$ and so $x\sim w\sim y$. Therefore, $\cup_{j=1,j\neq i}^{n}\gamma_j\subseteq A$. Hence,  $T+\cup_{j=1,j\neq i}^{n}\gamma_j\subseteq A$, which contradicts the assumption that $A$ is a cut set.

 Now suppose  $x\in \gamma_i$ for some $i\in \{1,2,\ldots,n\}$ and $y\in T$. Then $|V(Rx)\cap \mathrm{Min}(R)|=1$  and $|V(Ry)\cap \mathrm{Min}(R)|\geq 2$. Thus, by Lemma \ref{edgejiolemmaREDUCE}, we have $x\sim y$. Therefore, $A$ is not a cut set. Finally, suppose $x\in T$ and $y\in T$, then $ |V(Rx)\cap  \mathrm{Min}(R)|\geq 2$ and $ |V(Ry)\cap  \mathrm{Min}(R)|\geq 2$. By Lemma \ref{edgejiolemmaREDUCE}, we have $x\sim y$. Therefore, $A$ can not be a cut set. Thus, $A\subseteq Z(R)^{*}\setminus\{\gamma_i\}$ is not a cut set. Hence, for each   $1\leq i\leq n$, $Z(R)^{*}\setminus \{\gamma_i\}$ is a cut set.
\end{proof}
         
\begin{corollary} For $i\in S$, the set $Z(R)^*\setminus\ \{\gamma_i\}$ is a cut set of $W\Gamma(R)$ of size $|Z(R)^*|-\alpha_{max}$.  Consequently, $\kappa(W\Gamma(R))\leq|Z(R)^*|-\alpha_{max}$. 
\end{corollary}
\begin{lemma}\label{degrrinreduce}
  For $i\in S$,  every vertex contained in the set $\gamma_i$ is of the degree $|Z(R)^*|-\alpha_{max}$. As a consequence,  $\delta(W\Gamma(R))\leq |Z(R)^*|-\alpha_{max}$.
\end{lemma}
\begin{proof} Assume that $ \mathrm{Min}(R)=\{\mathfrak{p_i}\mid i\in\{1,2,\ldots,n\}\}$.
    Let $i\in S$ and $x\in \gamma_i$. Then $V(Rx)\cap \mathrm{Min}(R))=\{\mathfrak{p_i}\}$.  Note that ${\rm ann}(x)=\cap_{j=1,j\neq i}^{n}\mathfrak{p_j}$.  Suppose that $y\in Z(R)^*$. Then $|V(Ry)\cap  \mathrm{Min}(R)|=k$. First, suppose $k\geq 2$. Then by Lemma \ref{edgejiolemmaREDUCE}, we have $x\sim y$.
      Now, suppose $k=1$ such that $|V(Ry)\cap  \mathrm{Min}(R)|=\{\mathfrak{p_j}\}$, where $j\in \{1,2, \ldots,n\}$. If $i\neq j$, then 
     by Lemma \ref{edgejiolemmaREDUCE}, we have $x\sim y$.  If $i=j$, then $V(Ry)\cap  \mathrm{Min}(R)=\{\mathfrak{p_i}\}$ and so  ${\rm ann}(y)=\cap_{j=1,j\neq i}^{n}\mathfrak{p_j}$.  Let $x'\in {\rm ann}(x)=\cap_{j=1,j\neq i}^{n}\mathfrak{p_j}$ and $y'\in {\rm ann}(y)=\cap_{j=1,j\neq i}^{n}\mathfrak{p_j}$. Then $x',y'\not\in\mathfrak{p_i}$ and so $x'y'\neq0 $. Therefore, $x\not\sim y$. Consequently,  $deg(x)=|Z(R)^*|-\alpha_{max}$.
      \end{proof}
   

The following theorem determines the vertex connectivity of the graph $W\Gamma(R)$. 
\begin{theorem}
  Let $R$ be a reduced ring such that $ |\mathrm{Min}(R)|=n\geq 3$. Then for $i\in S$, $Z(R)^*\setminus\ \{\gamma_i\}$ is a minimum cut set of $ W\Gamma(R)$. If $R$ is finite, then $$\kappa(W\Gamma(R))=|Z(R)^*|-\alpha_{max}=\delta(W\Gamma(R)).$$
\end{theorem}
\begin{proof} Let $\mathrm{Min}(R)=\{\mathfrak{p_i}\mid i\in\{1,2,\ldots,n\}\}$.
By Theorem \ref{cutsetinreduced}, for every $i\in S$, $Z(R)^*\setminus\ \{\gamma_i\}$ is a cut set of $W\Gamma(R)$. Let $X$ be the minimum cut set of $W\Gamma(R)$. Then $X$ separates the graph into two subgraphs, say $A$ and $B$. Let $x\in A$ and  $y\in B$. Then  by Lemma \ref{edgejiolemmaREDUCE},  we have $V(Rx)\cap \mathrm{Min}(R)=V(Ry)\cap \mathrm{Min}(R)$ and $|V(Rx)\cap \mathrm{Min}(R)|=1$. Without loss of generality, assume that  $V(Rx)\cap\mathrm{Min}(R)=\{\mathfrak{p_i}\}=V(Ry)\cap \mathrm{Min}(R)$.  Now define the set $T=\{x\in Z(R)^*: |V(Rx)\cap \mathrm{Min}(R)|\geq2\}$. By Lemma \ref{vertesetcondition}, we have $Z(R)^*=T+\cup_{i=1}^{n}\gamma_i$. For any $z\in T$, we have $|V(Rz)\cap \mathrm{Min}(R)|\geq2$. Then by Lemma \ref{edgejiolemmaREDUCE}, we get $x\sim z\sim y$. Thus, $z\in X$ and so $T\subseteq X$. Now, let $w\in \gamma_i$, where $i\in\{1,2, \ldots,n\}$. For $i\neq j$. observe that  $V(Rw)\cap\mathrm{Min}(R)=\{\mathfrak{p_j}\}$. By Lemma \ref{edgejiolemmaREDUCE}, we obtain $x\sim w$. Similarly, $y\sim w$.
Therefore,  $Z(R)^{*}\setminus\{\gamma_i\}$ is contained in  $ X$ and so $|Z(R)^*|-\alpha_{max}\geq \kappa(W\Gamma(R))=|X|\geq |Z(R)^{*}\setminus \{\gamma_i\})|\geq  |Z(R)^*|-\alpha_{max}$. This gives $\kappa(W\Gamma(R))= |Z(R)^*|-\alpha_{max}$. This completes our proof.
\end{proof}

\begin{example} Let  $R=\mathbb{Z}_3 \times \mathbb{Z}_6$ be a reduced ring. The zero-divisor graph of  $R$ has been considered in {\cite[Example 3.7]{cote2011cut}}. 
The minimal prime ideals of $R$ are 
$$
\mathfrak{p}_1 = \{(0,0),(0,1),(0,2),(0,3),(0,4),(0,5)\}$$
$$\mathfrak{p}_2 = \{(0,0),(1,0),(2,0),(0,2),(1,2),(2,2)\}$$
$$\mathfrak{p}_3 = \{(0,0),(1,0),(2,0),(0,3),(1,3),(2,3)\}
.$$ Additionally, ${\gamma_1=\{(0,1),(0,5)\}}$, ${\gamma_2=\{(1,2),(2,2),\allowbreak(1,4),(2,4)\}}$ and ${\gamma_3=\{(1,3),(2,3)\}}$. By Fig. 1, the cut sets of $W\Gamma(R)$ are  $A_1=\{(0,2),(0,3),(0,4),(1,0),(1,2),(1,3),(1,4),\allowbreak (2,0), (2,2), (2,3),(2,4)\}$,  $A_2=\{(0,1),(0,2),(0,3),(0,4),(0,5),(1,0),(1,3),(2,0),(2,3)\} $ and $A_3=\{(0,1),(0,2),(0,3),(0,4),(0,5),(1,0),(1,2),(1,4),(2,0),(2,2), (2,4)\}$. One can observe that $A_1 =Z(R)^*\setminus\{\gamma_1\}$, $A_2=Z(R)^*\setminus\{\gamma_2\}$ and $A_3 =Z(R)^*\setminus\{\gamma_3\}$. Moreover,  $\kappa(W\Gamma(R))=|A_2|=|Z(R)^*\setminus\{\gamma_2\}|=9.$
  
\end{example}

  \begin{figure}[h!]
 			\centering
 			\includegraphics[width=0.7 \textwidth]{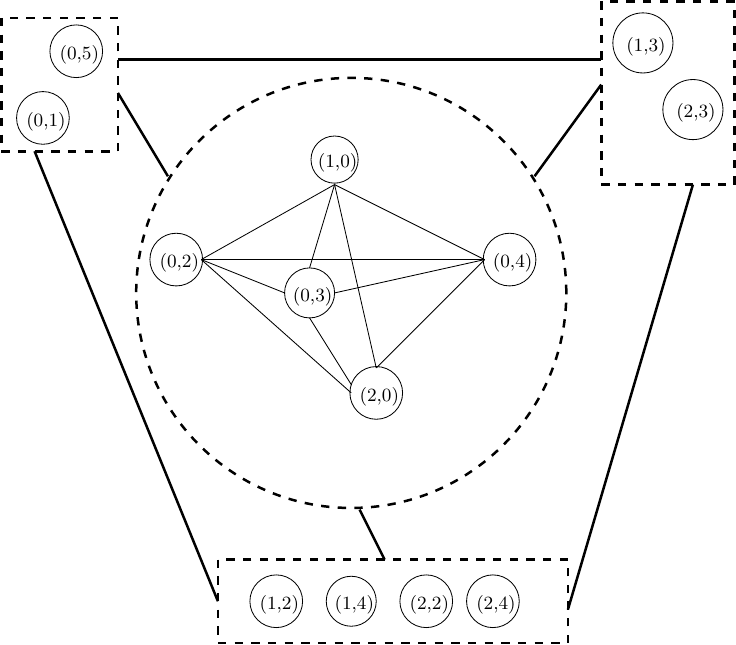}
 			\caption{ $\textbf{}$ $W\Gamma(\mathbb{Z}_3\times \mathbb{Z}_6)$ }
    \label{fig2.png}
 \end{figure}

\begin{example}
 Let    $
R=\mathbb{Z}_3\times\mathbb{Z}_3\times\mathbb{Z}_5\times\mathbb{Z}_7
$ be a reduced ring. Then the minimal prime ideals of \(R\) are $
\mathfrak{p_1}=\{0\}\times\mathbb{Z}_3\times\mathbb{Z}_5\times\mathbb{Z}_7$,
$\mathfrak{p_2}=\mathbb{Z}_3\times\{0\}\times\mathbb{Z}_5\times\mathbb{Z}_7$,
$\mathfrak{p_3}=\mathbb{Z}_3\times\mathbb{Z}_3\times\{0\}\times\mathbb{Z}_7$,
$\mathfrak{p_4}=\mathbb{Z}_3\times\mathbb{Z}_3\times\mathbb{Z}_5\times\{0\}.
$ For $1\leq i\leq 4$, $\gamma_i=\mathfrak{p_i}\setminus\cup_{j\neq i}\mathfrak{p_j}$.
Moreover,
$\gamma_1
=\{(0,a,b,c):a\in\mathbb{Z}_3^*,\,
b\in\mathbb{Z}_5^*,\,
c\in\mathbb{Z}_7^*\},
$$\gamma_2=\{(a,0,b,c):a\in\mathbb{Z}_3^*,\,
b\in\mathbb{Z}_5^*,\,
c\in\mathbb{Z}_7^*\},$$\gamma_3
=\{(a,b,0,c):a,b\in\mathbb{Z}_3^*,\,
c\in\mathbb{Z}_7^*\}
$ and $\gamma_4
=\{(a,b,c,0):a,b\in\mathbb{Z}_3^*,\,
c\in\mathbb{Z}_5^*\}.$ Note that $\alpha_{max}=|\gamma_1|= |\gamma_2|$. Therefore, $Z(R)^*\setminus\{\gamma_1\}$ and  $Z(R)^*\setminus\{\gamma_2\}$ are the minimum cut sets and so $\kappa(W\Gamma(R))= |Z(R)^*|-|\gamma_1|=218-48=170$.
\end{example}

  \begin{lemma}\label{minimumdegree} Suppose $R$ is a finite reduced ring such that $ |\mathrm{Min}(R)|=n\geq 3$. Let $x\in Z(R)^*$ such that $deg(x)=\delta(W\Gamma({R})$. Then $|V(Rx)\cap  \mathrm{Min}(R)|=1$.
    \end{lemma}
\begin{proof} Let $ \mathrm{Min}(R)=\{\mathfrak{p_i}\mid i\in\{1,2,\ldots,n\}\}$. On the contrary, suppose that  $|V(Rx)\cap \mathrm{Min}(R)|=k\geq2$. 
Let $y\in Z(R)^*$. Then $|V(Ry)\cap \mathrm{Min}(R)|\geq1$ and so  by Lemma \ref{edgejiolemmaREDUCE}, we have $x\sim y.$
Consequently, $deg(x)= |Z(R)^*|-1> \delta(W\Gamma(R))$, a contradiction. Hence, $|V(Rx)\cap \mathrm{Min}(R)|=1$.
\end{proof}

The following theorem characterizes all the vertices of minimum degree of the graph $W\Gamma(R)$.
\begin{theorem}\label{minidegreevertexreduce}
     Let $R$ be a finite reduced ring with  $ |\mathrm{Min}(R)|=n\geq 3$. Then $deg(x)=\delta(W\Gamma(R))$ if and only if $x\in \gamma_i$ and $i\in S$.
  \end{theorem}
\begin{proof}
    Let $ \mathrm{Min}(R)=\{\mathfrak{p_i}\mid i\in\{1,2,\ldots, n\}\}$. First, suppose $x\in \gamma_i$ and $i\in S$. Then $deg(x)=|Z(R)^*|-\alpha_{max}=\delta(W\Gamma(R))$. Let $x\in\ Z(R)^*$ such that $deg(x)=\delta(W\Gamma(R))$. Then by Lemma \ref{minimumdegree} and by Lemma \ref{vertesetcondition}, we have $x\in\gamma_i$.
    Now, suppose $i\not\in S$. Then we have  $|\gamma_i|<\alpha_{max}$. Since $x\in \mathfrak{\gamma_i}$ and  $i\not\in S$, we have $deg(x)=|Z(R)^*|-|\gamma_i|> |Z(R)^*|-\alpha_{max}= \delta(W\Gamma(R))$, a contradiction. Thus, $i\in S$.
\end{proof}

\section{Artinian rings} 
Let $R$ be an Artinian ring. Then by  {\cite[Theorem 3.1.4]{bini2002finite}}, $R \cong R_1 \times   R_2\times \cdots \times R_n$, where $R_i$ is an Artinian local ring for every $1 \leq i \leq n$. If each $R_i$ is a non-reduced Artinian local ring, then by Theorem \ref{nikmehr2021weaklycmpt}, the graph $W\Gamma(R)$ is complete. Therefore, in this section, we consider   $R \cong F_1\times F_2\times\cdots\times F_k\times R_{k+1}\times\cdots\times R_n$,
where $F_i$ is a field for every $1 \leq i \leq k$, and $R_j$ is an Artinian local ring satisfying $|\operatorname{Nil}(R_j)^{*}| \neq 0$ for every $k+1 \leq j \leq n$. We write each element $x \in R$ by an ordered  tuple $(x_1,x_2, \ldots, x_k, x_{k+1},x_{k+2}, \ldots, x_n)$.
 For $1\leq i\leq k$, define the sets
$$H_i = \{x\in R \mid x_i = 0 \text{ for exactly one } i,\ \text{and } x_j \in U( R_j) \text{ for all } j(k+1\leq j\leq n)\}.$$

 To prove our main result of this section, the following lemma is essential.

\begin{lemma}\label{nbdlemmaartinian}
Let $R \cong F_1\times F_2\times\cdots\times F_k\times R_{k+1}\times\cdots\times R_n$, where each $F_i$ is a field and every $R_j$ is a non-reduced Artinian local ring. Then for $y\in H_i$ and $x\notin H_j$ for all $j$, $(1\leq j\leq k)$, we have $N(y)\subseteq N(x)$.
 \end{lemma}
\begin{proof} Let $x=(x_1,x_2,\ldots,x_n)$ and $y=(y_1,y_2,\ldots,y_n)$ such that $y\in H_i$ and $x\notin H_j$ for all $j$, $(1\leq j\leq k)$. Since $y\in H_i$, it follows that ${\rm ann}(y)=\{(0,0,\ldots, u_i,\ldots, 0)\mid u_i\in F_i^*\})$. Suppose $z\in N(y)$. Then there exists $w=(w_1,w_2,\ldots,w_k\allowbreak,\ldots,w_n)\in {\rm ann}(z)$ and $x'\in {\rm ann}(y)$ such that $wx'=0$. Consequently, $w_i=0$.
 Since $x\not\in H_j$, for all $j(1\leq j\leq k)$, then we have the following two cases.
 
 \textbf{Case-1}   $x_j=0$ for at least two $j\in \{1,2, \ldots,k\}$. Without loss of generality, assume that $x_s=0$ and $x_t=0$. Let $v=(0,0,\ldots,v_s,0,0,\ldots,0)$ and $v'=(0,0,\ldots,v_t,0,0,\ldots,0) $. Note that $v,v'\in {\rm ann}(x)$. If either $w_s=0$ or $ w_t=0$, then either $vw=0$ or $v'w=0$. Therefore  $z\in N(x)$. Suppose  $w_s\neq0$ and $ w_t\neq 0$. Observe that  $b=(w_1,w_2,\ldots,w_{s-1},0,w_{s+1}\allowbreak\ldots,w_n)\in {\rm ann}(z)$ and $b'=(w_1,w_2,\ldots,w_{t-1},0,w_{t+1}\allowbreak,\ldots,w_n)\in {\rm ann}(z)$. Clearly, either $vb=0$ or  $v'b'=0$. Thus, $z\in N(x)$.

 \textbf{Case-2}  $x_l\notin U(R_l)$ for some $l\in\{k+1,\ldots,n\}$. Without loss of generality, assume that  $x_{k+1}\in \mathrm{Nil}(R_{k+1})$. First suppose that $x_{k+1}=0$. Then $c=(0,0,\ldots,v_{k+1},0,0,\ldots,0)\in {\rm ann}(x) $. Notice that $w'=(w_1,w_2,\ldots,w_{k},0,w_{k+2}\allowbreak\ldots,w_n)\in {\rm ann}(z)$. Clearly $cw'=0$ and so  $z\in N(x)$. Now suppose  $x_{k+1}\in \mathrm{Nil}(R_{k+1})^*$ and let $q$ be the index of nilpotence of $x_{k+1}$. Note that  $c'=(0,0,\ldots,{x_{k+1}}^{q-1},0,\ldots,0) \in {\rm ann}(x)$. If $w_{k+1}=0$, then  $wc'=0 $. Thus, $z\in N(x)$.
   If $w_{k+1}\neq 0$, then either $w_{k+1}\in \mathrm{Nil}(R_{k+1})^*$ or $w_{k+1}\in U(R_{k+1})$. First, suppose  $w_{k+1}\in \mathrm{Nil}(R_{k+1})^*$ and $r$ be the index of nilpotence of $w_{k+1}$.   Then clearly,  $d=(w_1,w_2,\ldots,\allowbreak w_k,{w_{k+1}}^{r-1},\ldots,w_n)\in {\rm ann}(z)$. If ${w_{k+1}}^{r-1}{x_{k+1}}^{q-1}=0$, then $dc'=0$ and so $z\in N(x)$. If ${w_{k+1}}^{r-1}{x_{k+1}}^{q-1}\neq0$. Consider $e=(0,0,\ldots{w_{k+1}}^{r-1}{x_{k+1}}^{q-1},0,\ldots,0)$. Clearly  $e \in {\rm ann}(x)$ and  $ed=0$. Therefore, $z\in N(x)$.
We now suppose that ${w_{k+1}}\in U(R_{k+1})$. Observe that
   $d'=(w_1,w_2,\ldots, w_k,{x_{k+1}}^{q-1},\ldots,w_n)\in {\rm ann}(z)$ and $d'c'=0$. Therefore,  $z\in N(x)$.
   \end{proof}

\begin{theorem}\label{finitering}
Let $R$ be an Artinian ring such that the weakly zero-divisor graph $W\Gamma(R)$ is non-empty. Then $\kappa(W\Gamma(R))= \delta(W\Gamma(R))$.
    \end{theorem}
    \begin{proof} By the structure theorem of Artinian rings,
$R \cong R_1 \times R_2 \times \cdots \times R_n,$
where $n \geq 1$ and each $R_i$ is an Artinian local ring.
If $n = 1$, then observe that the set  $\mathrm{Nil}(R)^* $ is non-empty. Otherwise, $W\Gamma(R)$ is empty. For $\mathrm{Nil}(R)^* \neq \phi$, by Theorem  \ref{nikmehr2021weaklycmpt}, we have $W\Gamma(R)$ is a complete graph and so $\kappa(W\Gamma(R)) = \delta(W\Gamma(R)).$
We may now suppose that $n\geq 2$. If  $\mathrm{Nil}(R_i)^* $ is non-empty for each $i(1\leq i\leq n)$, then again by Theorem \ref{nikmehr2021weaklycmpt}, the graph $W\Gamma(R)$ is complete and so $\kappa(W\Gamma(R)) = \delta(W\Gamma(R)).$

 Now suppose  $\mathrm{Nil}(R_i)^* $ is empty for some $i$, where $1\leq i\leq n$. Without loss of generality, assume that $\mathrm{Nil}(R_i)^*=(0)$ for all $i$, where $1\leq i\leq k$ and  $\mathrm{Nil}(R_i)^*\neq (0) $ for every $j$, where $k+1\leq j\leq n$. In view of the Proposition \ref{minimumdegree+vertex}, we show that $\kappa(W\Gamma(R)) \geq \delta(W\Gamma(R))$. In order to prove this, we show that for  $A \subsetneq  V(W\Gamma(R))$ such that $|A|< \delta(W\Gamma(R))$, the graph  $H = W\Gamma(R) - A$ is connected. For $1\leq i\leq k$, consider $$H_i = \{x \mid x_i = 0 \text{ for exactly one } i(1\leq i\leq k),\ \text{and } x_j \in U( R_j) \text{ for all } j(k+1\leq j\leq n)\}.$$
 Let $x,y\in H$. In the following cases, we provide at least one path between $x$ and $y$.

 \textbf{Case-1}  $x\notin H_i$ and $y\notin H_j$ for all $1\leq i,j\leq k$. Consider a vertex $x'=(x_1,x_2,\ldots,x_{i-1},\allowbreak 0,x_{i+1},\ldots,x_k,$
$x_{k+1},x_{k+2},\ldots,x_n)$ of $H_i$. Then ${\rm ann}(x')=\{(0,0,\dots,u_i,\allowbreak 0,\dots,0)\mid u_i\in F_i^*\}$. Let $z\in H_i$. Then ${\rm ann}(z)=\{(0,0,\dots,u_i,0,\dots,0)\mid u_i\in F_i^*\}$. Observe that for any $x''\in {\rm ann}(x')$ and $z'\in {\rm ann}(z)$, we have $x''z'\neq0$ and so $x'\not\sim z$. Therefore, $N(x')\subseteq Z(R)^*-H_i$. It follows that $|Z(R)^*-H_i|\geq \delta (W\Gamma(R))> |A|$ and so  $H$ must contain a vertex $u\in N(x')$. Since by Lemma \ref{nbdlemmaartinian}$, N(x')\subseteq N(x)$. It follows that $u\in N(x)$. 
Similarly for $y\notin H_j$ for all $j$ ($1\leq j\leq k$), we obtain a vertex $y'\in H_j$ such that $N(y')\subseteq Z(R)^*-H_j$. Indeed, $H$ must contains a vertex $v\in N(y)$. With the help of $u$ and $v$, in the following subcases, we find at least one path between $x$ and $y$.

\textbf{Subcase-1.1}
 $u\in H_m$ and  $v\in H_l$. Since $u\in N(x')$ and $v\in N(y')$, we have $m\neq i$ and $j\neq l$. First, suppose that $m\neq l$. For $u'=(0,0,\ldots,0,u'_m,\ldots,0)\in {\rm ann}(u)$ and $v'=(0,0,\ldots,0,v'_l,\ldots,0)\in {\rm ann}(v)$, we obtain $u'v'=0$ and so $u\sim v$. Consequently, we get a path $x\sim u\sim v\sim y$.
 
We may now suppose that  $m=l$. Since $u\in H_l$, we have $u'=(0,0,\ldots,0,u'_l,\ldots,0
)\in {\rm ann}(u)$. Since $y\not\in H_j$, it follows that either $y_i=0$ for at least two $i\in \{1,2, \ldots,k\}$ or $y_j\notin U(R_j)$ for some $j\in\{k+1,\ldots,n\}$. Let $y_i=0$ for at least two $i\in \{1,2, \ldots,k\}$. Without loss of generality, assume that $y_g=0$ and $y_h=0$, where $g,h\in\{1,2, \ldots,k\}$. Then note that either $g\neq l$ or $h\neq l$. If  $g\neq l$, then $y'=(0,0,\ldots,y_g',0,\ldots,0)\in {\rm ann}(y)$ and so $u'y'=0$. Thus, $u\sim y$. Moreover, $x\sim u\sim y$. If  $h\neq l$, then $y''=(0,0,\ldots,y_h',0,\ldots,0)\in {\rm ann}(y)$ and so $u'y''=0$. Therefore, $u\sim y$. It follows that  $x\sim u\sim y$.
Now suppose $y_j\not\in U(R_j)$ for some $j\in\{k+1,\ldots, n\}$. It follows that $y_j\in \mathrm{Nil}(R_j)$ for some $j\in\{k+1,\ldots, n\}$. Without loss of generality, assume that $y_{k+1}\in \mathrm{Nil}(R_{k+1})$. First suppose that $y_{k+1}=0$. Since $\mathrm{Nil}(R_{k+1})^*\neq \{0\}$, it follows that there exists $z_{k+1}\in \mathrm{Nil}(R_{k+1})^*$. Let  $y'=(0,0,\ldots,0, {z_{k+1}},0,\dots,0)\in {\rm ann}(y)$. Then note that  $u'y'=0$. Thus,  $u\sim y$ and so  $x\sim u\sim y$. Now  suppose that $y_{k+1}=0$ and $q$ be the index of nilpotence of $y_{k+1}$. Then $y''=(0,0,\ldots,0, {y_{k+1}}^{q-1},0,\dots,0)\in {\rm ann}(y)$. Also $u'y''=0$. Thus,  $u\sim y$. Therefore,  $x\sim u\sim y$.

\textbf{Subcase-1.2}
 $u\in H_l$ and $v\notin H_j$ for all $j$, where $1\leq j\leq k$. Since $v\notin H_j$ for all $j$, then either $v_i=0$ for at least two $i\in \{1,2,\ldots,k\}$ or $v_j\not\in U(R_j)$ for some $j\in\{k+1,k+2,\ldots,n\}$.  By the same argument used in \textbf{Subcase-1.1}, there always exist $v'\in{\rm ann}(v)$ and $u'\in {\rm ann}(u)$ such that $u'v'=0$ and so  $u\sim v$. We obtain a path $x\sim u\sim v\sim y$.

\textbf{Subcase-1.3}
 $u\notin H_i$ and $v\notin H_j$ for all $i,j $, where $1\leq i,j\leq k$.  Then either $u_i=0$ for at least two $i\in \{1,2,\ldots,k\}$ or $u_j\not\in U(R_j)$ for some $j\in\{k+1,k+2,\ldots,n\}$, and either $v_i=0$ for at least two $i\in \{1,2,\ldots,k\}$ or $v_j\not\in U(R_j)$ for some $j\in\{k+1,k+2,\ldots,n\}$. First suppose that $u_i=0$ for at least two $i\in \{1,2,\ldots,k\}$.
Without loss of generality, assume that $u_g=0$ and $u_h=0$, where $g,h\in \{1,2,\dots,k\}$. If $v_i=0$ for at least two $i\in \{1,2,\ldots,k\}$, then without loss of generality assume that  $v_s=0$ and $v_t=0$, where $s,t\in \{1,2,\dots,k\}$.

  Observe that both $g$ and $h$ can not be equal to either $s$ or $t$.
If $g\neq s$,
then there exist $u'=(0,0,\ldots,u_g',0\ldots,0)\in {\rm ann}(u)$ and $v'=(0,0,\ldots,v_s',0\ldots,0)\in {\rm ann}(v)$ such that $u'v'=0$ and so  $u\sim v$. 
If $g\neq t$, then there exist $u'=(0,0,\ldots,u_g',0\ldots,0)\in {\rm ann}(u)$ and $v'=(0,0,\ldots,v_t',0\ldots,0)\in {\rm ann}(v)$ such that $u'v'=0$ and so  $u\sim v$. If $h\neq s$ or  $h\neq t$, then by the similar argument used above $u\sim v$. Suppose $g=s$  and $h=t$. Then  $u''=(0,0,\ldots,u_g'',0\ldots,0)\in {\rm ann}(u)$ and $v''=(0,0,\ldots,v_t'',0\ldots,0)\in {\rm ann}(v)$ and also $u''v''=0$. Thus, $u\sim v$.

If $v_j\not\in U(R_j)$ for some $j\in\{k+1,k+2,\ldots,n\}$, then by the similar argument used in \textbf{Subcase-1.1}, there exists $ v''\in {\rm ann}(v)$ such that $u'v''=0$ and so $u\sim v$. Now, suppose $u_j\not\in U(R_j)$ for some $j\in\{k+1,k+2,\ldots,n\}$ . If $v_i=0$ for at least two $i\in \{1,2,\ldots,k\}$, then again by the similar argument used in \textbf{Subcase-1.1}, we have  $u\sim v$. If $v_j\not\in U(R_j)$ for some $j\in\{k+1,k+2,\ldots,n\}$, then  $v_j\in \mathrm{Nil}(R_j)$  for some $j\in\{k+1,\ldots,n\}$. Without loss of generality, assume that $v_s\in \mathrm{Nil}(R_s)$ and $u_t\in \mathrm{Nil}(R_t)$. Let $q$ and $r$ be the index of nilpotence of $v_s$ and $u_t$, respectively. Then $v'=(0,0,\ldots,0,{v'_s}^{q-1},0,\ldots,0)\in{\rm ann}(v)$ and $u'=(0,0,\ldots,0,{u'_t}^{r-1},0,\ldots,0)\in {\rm ann}(u)$. If $s\neq t$, then $u'v'=0$. Now let  $s= t$. If  ${u'_t}^{r-1}{v'_s}^{q-1}=0$, then $u\sim v$. If ${u'_t}^{r-1}{v'_s}^{q-1}\neq0$, then  $a = (0, \ldots, 0, {u'_t}^{r-1}{v'_s}^{q-1}, 0, \ldots, 0) \in \operatorname{ann}(u) \cap \operatorname{ann}(v)$.  Clearly,  $a^2=0$ and so $u\sim v$.

\textbf{Case-2}
 $x\notin H_i$ for all $i$, where $i\in\{1,2,\ldots,k\}$ and $y\in H_j$ for some $j\in\{1,2,\ldots,k\}$. In similar lines of \textbf{Case-1}, we get $N(y)\subseteq Z(R)^*-H_j$. It implies that  $|Z(R)^*-H_j|\geq \delta(W\Gamma(R))> |A|$ and so that $H$ must contain a vertex $v'\in N(y)$. Since $x\notin H_i$, by \textbf{Case-1}, there exists $u\in H$ such that $u\in {N(x)}$.
By the similar argument used in \textbf{Case-1}, we obtain $u\sim v'$. Thus,  $x\sim u\sim v'\sim y$ is a path in $H$.

\textbf{Case-3}
 $x\in H_i$ and $y\in H_j$ for some $i,j$, where $i,j\in \{1,2,\ldots,k\}$. Then $N(x)\subseteq Z(R)^*-H_i$ and $N(y)\subseteq Z(R)^*-H_j$ for some $i,j\in \{1,2,\ldots,k\}$. Consequently, we have $|Z(R)^*-H_i|\geq \delta(W\Gamma(R))> |A|$ and $|Z(R)^*-H_j|\geq \delta(W\Gamma(R))> |A|$. Moreover, $H$  must contain  vertices $u'\in N(x)$ and  $v'\in N(y)$. By the similar argument used in  \textbf{Case-1}, we obtain  paths  $x\sim u'\sim v'\sim y$ or  $x\sim u'\sim y$ or  $x\sim v'\sim y$  between $x$ and $y$. This completes our proof.
 \end{proof}

The following result determines the cut sets of  $W\Gamma(R)$.
\begin{theorem}\label{cutsetartinianring} 
Let $R$ be an Artinian ring. Then for each $i\in\{1,2,\dots k\}$,  the set $Z(R)^*\setminus\{H_i\}$ is a cut set of $W\Gamma(R)$.
\end{theorem}
 \begin{proof}   Let  $R \cong F_1 \times F_2\times \cdots \times F_k \times R_{k+1}\times \cdots \times R_n$, where $F_i$ is a field for every $1 \le i \le k$ and $|\operatorname{Nil}(R_j)^{*}| \neq 0$ for every $k+1 \leq j \leq n$.
 We claim that the graph $W\Gamma(H_i)$ induced by the set  $H_i\subseteq Z(R)^*$ is a null graph with only neighbourhood to the set $Z(R)^*\setminus \{H_i\}$.
  Let  $x,y\in H_i$. Then ${\rm ann}(x)={\rm ann}(y)=\{(0,0,\dots,u_i,0,\dots,0)\mid u_i\in F_i\}$.  For any  $x'\in {\rm ann}(x)$ and $y'\in {\rm ann}(y)$, note that  $x'y'\neq0$ and so $x\not\sim y$. Let $s=(s_1,s_2,\dots,s_k,s_{k+1},\dots,s_n)\in Z(R)^*\setminus \{H_i\}$. Then either $s_j=0$ for some $j\in\{1,2,\ldots,k\}\setminus\{i\}$ or $s_l\in \mathrm{Nil}(R_j)$ for some $l\in\{k+1,k+2,\ldots,n\}$. First suppose  $s_j=0$ for some $j\in\{1,2,\ldots,k\}\setminus\{i\}$. 
  Without loss of generality, assume that $s_p=0$, where  $p\in\{1,2,\ldots,k\}\setminus\{i\}$. Then note that $\{(0,0,\dots,u_p,0,\dots,0)\mid u_p\in F_p\}\subseteq{\rm ann}(s)$. Let 
 $s'\in\{(0,0,\dots,u_p,0,\dots,0)\mid u_p\in F_p\}\subseteq{\rm ann}(s)$ and $x''\in \{(0,0,\dots,u_i,0,\dots,0)\mid u_i\in F_i\}={\rm ann}(x)$.  Then clearly $s'x''=0$ and so $x\sim s$. Now suppose that  $s_j\in \mathrm{Nil}(R_j)$ for some $j\in\{k+1,k+2,\ldots,n\}$. Without loss of generality, assume that
   $s_q\in\mathrm{Nil}(R_q)$, where  $q\in\{k+1,k+2,\ldots,n\}$. First, suppose that $s_q=0$. Since $\mathrm{Nil}(R_q)^*\neq(0)$, notice that $t=(0,0,\dots,0,0,\dots,t_q,0,\ldots,0)\in{\rm ann}(s)$, where $t_q\in \mathrm{Nil}(R_q)^*$. Observe that $tx''=0$. Now suppose that $s_q\in\mathrm{Nil}(R_q)^*$ and let $r$ be the index of nilpotence of $s_q$. Then $s''= (0,0,\dots,0,0,\dots,s_q^{r-1},0,\ldots,0)\in{\rm ann}(s)$. Clearly $s''x''=0$ and so $s\sim x$. Therefore, $Z(R)^*\setminus\{H_i\}\subseteq N(H_i)$. Since $W\Gamma(H_i)$ is a null graph, we have $ N(H_i)=Z(R)^*\setminus\{H_i\}$. Moreover, $Z(R)^*\setminus\{H_i\}$ isolates the vertices of the graph $W\Gamma(H_i)$.
   
 Let $S=Z(R)^*\setminus\{H_i\}$.
 Suppose $A\subsetneq S$ such that the graph $W\Gamma(Z(R)^*\setminus A) $ is disconnected. Since $A\subsetneq S$, it follows that there exists $a\in T=S\setminus A$. Observe that $Z(R)^*\setminus A= H_i\cup T$. To end the proof, we show that the graph $W\Gamma( H_i\cup T)$ is connected. Now let $x,y\in W\Gamma( H_i\cup T)$ such that $x\not\sim y$. Then, either $x,y\in H_i$ or $x,y\in T$. Otherwise, $x\sim y$. 
  First suppose $x,y\in H_i$.  Then again, by the same argument used earlier, there exists $t\in T$ such that $x\sim t\sim y$. We now suppose that $x,y \in T$. Then, by the argument used earlier, there exists $h\in H_i$ such that $x\sim h\sim y$. This completes our proof.
\end{proof}
      
Let  ${\theta_{max}}:=\text{max}\{|H_i| : \text{ for all }1\leq i\leq k\}$ and $T:=\{i :  |H_i|=\theta_{max}\}.$
      
\begin{corollary}\label{cutsetseizeartinian} For $i\in T$, the set $Z(R)^*\setminus\{H_i\}$ is a cut set of $W\Gamma(R)$ of size $|Z(R)^*|-\theta_{max}$.  In particular, $\kappa(W\Gamma(R))\leq |Z(R)^*|-\theta_{max}$. 
\end{corollary}

\begin{corollary}\label{degreeinartinian}
     Let $R$ be an Artinian ring. Then every vertex of the set $H_i$ is of degree  $|Z(R)^*|-\theta_{max}$. As a consequence, $\delta(W\Gamma(R))\leq |Z(R)^*|-\theta_{max}$.
   \end{corollary}
\begin{proof} Let $x\in H_i$. By the argument used in the proof of Theorem \ref{cutsetartinianring}, we obtain $N(x)=Z(R)^*\setminus\{H_i\}$. This completes our proof because $i\in T$.
\end{proof}

The following result determines the vertex connectivity of the graph $W(\Gamma(R)$.

\begin{theorem}
    Let $R$ be an Artinian ring. Then for  $i\in T$, $Z(R)^*\setminus\{H_i\}$ is a minimum cut set of $W\Gamma(R)$. If $R$ is finite, then $$\kappa(W\Gamma(R))=|Z(R)^*|-\theta_{max}=\delta(W\Gamma(R)).$$
\end{theorem}
\begin{proof}
   By Corollary \ref{cutsetseizeartinian}, for all $i\in T$,  $Z(R)^*\setminus\{H_i\}$ is the cut set of $W\Gamma(R)$ of size  $|Z(R)^*|-\theta_{max}$. Let $X$ be a minimum cut set of $W\Gamma(R)$, and let $A$ and $B$ be two components of $W\Gamma(R)$. Define a set
   $$S=\{x\mid x_i = 0 \text{ for atleast two  } i(1\leq i\leq k) \text{ or }  x_j \notin U(R_j) \text{ for some } j(k+1\leq j\leq n).$$
   Clearly, $S= Z(R)^*\setminus\cup_{i=1}^{k}H_i$ and so $V(W\Gamma(R))=S+\cup_{i=1}^{k}H_i$. Suppose $a\in A$ and $b\in B$. Since $A$ and $B$ are the two separated sets of $W\Gamma(R)$, we obtain $a,b\in H_i$, where  $i\in \{1,2,\ldots,k\}$. Otherwise, $a\sim b$.  Let   $x\in S$. Then, by the same argument used in \textbf{Subcase-1.2} of Theorem \ref{finitering}, we obtain $a\sim x\sim b$.   Thus, $x\in X$ and so $S\subseteq X$.
 Now, let $y\in H_j$ for some $j\in\{1,\ldots,k\}\setminus\{i\}$. Then by the similar argument used in \textbf{Case-3} of Theorem \ref{finitering}, we get $a\sim x\sim b$. It follows that $y\in X$ and so $\cup_{j=1,j\neq i}^{k}H_j\subseteq X$.  Consequently, $S+\cup_{j=1,j\neq i}^{k}H_j\subseteq X$. Since $X$ is a minimum cut set and  $Z(R)^*\setminus\{H_i\}$, where  $i\in T$ is a cut set having cardinality $|Z(R)^*|-\theta_{max}$, we have 
 $|Z(R)^*|-\theta_{max}\geq \kappa(W\Gamma(R))=|X|\geq |Z(R)^{*}\setminus\{H_i\}|\geq|Z(R)^*|-\theta_{max}$. Thus, $\kappa(W\Gamma(R))=|Z(R)^*|-\theta_{max}$.
\end{proof} 

\begin{example}
Let $R=\mathbb{Z}_2\times\mathbb{Z}_3\times\mathbb{Z}_4\times\mathbb{Z}_8$ be an Artinian  ring.
Note that $$H_1=\{(0,a,b,c)\mid a\in\mathbb{Z}_3^*,\ b\in U(\mathbb{Z}_4),\ c\in U(\mathbb{Z}_8)\}$$ $$H_2=\{(a,0,b,c)\mid a\in\mathbb{Z}_2^*,\ b\in U(\mathbb{Z}_4),\ c\in U(\mathbb{Z}_8)\}.$$ Moreover, $|H_1|=16$ and $|H_2|=8$ and
$\theta_{max}= \text{max}\{|H_1|,|H_2|\}=16.$
Note that the sets $Z(R)^*\setminus\{H_1\}$ and $Z(R)^*\setminus\{H_2\}$ are the cut sets of $W\Gamma(R)$. Since $|H_1|> |H_2|$, it follows that  $Z(R)^*\setminus\{H_1\}$ is a minimum cut set of $W\Gamma(R)$. Consequently, $\kappa(W\Gamma(R))=|Z(R)^*|-\theta_{max}=192-16=176=\delta(W\Gamma(R)).$
\end{example}


 In contrast to Theorem ** , the following theorem characterizes all the vertices of minimum degree of the graph $W\Gamma(R)$, where $R$ is a finite non-reduced ring.
 \begin{theorem} 
  Let $R$ be finite non-reduced ring. Then $deg(x)=|Z(R)^*|-\theta_{max}=\delta(W\Gamma(R))$ if and only if $x\in H_i$ and $i\in T$.
 \end{theorem}
 \begin{proof}
      First, assume that $x\in H_i$ for some  $i\in T$. Then by Corollary \ref{degreeinartinian}. $deg(x)=|Z(R)^*|-\theta_{max}=\delta(W\Gamma(R))$.
 Now assume that $deg(x)=\delta(W\Gamma(R))$.
Suppose $x\in Z(R)^*$ such that  $x\not\in H_i$ for all $1\leq i\leq k$. 
For $y\in Z(R)^*$, note that either $y\in H_i$ for some $1\leq i\leq k$ or $y\notin H_j$ for all $1\leq j\leq k$.  By the same argument used in \textbf{Subcase-1.2} and  \textbf{Subcase-1.3} of Theorem \ref{finitering}, we have $x\sim y$.
Therefore, $deg(x)=|Z(R)^*|-1>|Z(R)^*|- \theta_{max}=\delta({W\Gamma(R)})$, a contradiction. Thus, $x\in H_i$ for some $i\in \{1,2, \ldots,k\}$. If $i\not\in T$, then $deg(x)=|Z(R)^*|-|H_i|> |Z(R)^*|-\theta_{max}=\delta({W\Gamma(R)})$, a contradiction. Therefore, $i\in T$
\end{proof} 
\section{Conclusion and Future Work} 
In this paper, we have investigated the vertex connectivity of the weakly zero-divisor graph $W\Gamma(R)$ associated with commutative rings. In this connection, we prove that, when $R$ is an Artinian ring or a reduced ring, the vertex connectivity of $W\Gamma(R)$ coincides with its minimum degree. In the case of reduced rings having finitely many minimal prime ideals, we obtained an explicit expression for the vertex connectivity in terms of the cardinalities of the sets determined by these minimal prime ideals. We also characterized the vertices of minimum degree through the sets associated with minimal prime ideals of maximum cardinality.
By using the structure theorem of Artinian rings, we ascertain minimum cut sets of $W\Gamma(R)$ and obtained its vertex connectivity. Indeed, it is equal to the minimum degree of $W\Gamma(R)$. 

Our results demonstrate a close relationship between the algebraic structure of the ring $R$ and the connectivity properties of its associated weakly zero-divisor graph. In particular, the decomposition of Artinian rings and the structure of minimal prime ideals in reduced rings are useful in describing the cut sets and determining the vertex connectivity of $W\Gamma(R)$. The following questions remain open for future research.
\medskip
\begin{problem}\label{op:general-n}
    Let $R$ be a reduced ring with an infinite number of minimal prime ideals. Is $\kappa(W\Gamma(R))=\delta(W\Gamma(R))$$?$
 \end{problem}
\begin{problem}\label{op:general-n}
  Classification of all the non-Artinian rings $R$ such that  $\kappa(W\Gamma(R))=\delta(W\Gamma(R))$$?$
 \end{problem}

\section*{ Acknowledgement} The first author gratefully acknowledges Birla Institute of Technology and Science (BITS) Pilani, India, for providing financial support. 
We are also thankful to Professor Akiyoshi Sannai, Department of Physics, Kyoto University, to correct the statement of Lemma \ref{minimumdegree} and Theorem \ref{minidegreevertexreduce} in an earlier version of this paper.

\section*{Declarations}
\noindent \textbf{Data Availability:}	There is no data associated with this article.

\medskip
\noindent \textbf{Funding:} The first author gratefully acknowledges Birla Institute of Technology and Science (BITS) Pilani, India, for providing financial support.

\medskip
\noindent \textbf{Conflict of interest:} The authors have no competing interests to declare that are relevant to the content of this article.


\end{document}